\documentclass{amsart}
\usepackage{amssymb, amsmath}

\title[Fractional minimal rank]{Fractional minimal rank}
 \author[B. W. Grossmann and H. J. Woerdeman]{Ben Grossmann}
\author[]{Hugo J.~Woerdeman}
\address{Department of Mathematics \\
Drexel University\\
3141 Chestnut St.\\
Philadelphia, PA, 19104}
\email{bwg25@drexel.edu,hugo@math.drexel.edu}

\thanks{HJW is supported by Simons Foundation grant 355645.}


\date{}

\newtheorem{thm}{Theorem}[section]

\newtheorem{claim}[thm]{Claim}

\newtheorem{cor}[thm]{Corollary}
\newtheorem{lem}[thm]{Lemma}

\newtheorem{prop}[thm]{Proposition}

\numberwithin{equation}{section}

\newcommand{\col}{\operatorname{col}}

\newcommand{\F}{\mathbb{F}}

\newcommand{\C}{\mathbb{C}}
\newcommand{\N}{\mathbb{N}}

\newcommand{\eps}{\varepsilon}
\newcommand{\s}[1]{\mathcal{#1}}

\DeclareMathOperator{\rk}{rank}
\DeclareMathOperator{\mr}{mr}
\DeclareMathOperator{\fmr}{fmr}
\DeclareMathOperator{\tmr}{tmr}
\DeclareMathOperator{\xmr}{xmr}

\DeclareMathOperator{\codim}{codim}

\renewcommand{\col}{\operatorname{col}}

\newcommand{\mat}[1]{\begin{bmatrix} #1 \end{bmatrix}}

\newcommand{\smat}[1]{
\left[\begin{smallmatrix} #1
\end{smallmatrix}\right]
}

\makeatletter
\newcommand{\vast}{\bBigg@{4}}
\newcommand{\Vast}{\bBigg@{5}}
\makeatother

\begin{document}

\begin{abstract}
The notion of fractional minimal rank of a partial matrix is introduced,  a quantity that lies between the triangular minimal rank and the minimal rank of a partial matrix. The fractional minimal rank of partial matrices whose bipartite graph is a minimal cycle are determined. Along the way, we determine the minimal rank of a partial block matrix with invertible given entries that lie on a minimal cycle. Some open questions are stated.
\end{abstract}

\maketitle

\section{Introduction}\label{sec:intro} In this paper we introduce the tensor product of a partial matrix ${\mathcal A}$ with a traditional matrix $B$ resulting in a partial block matrix in a standard way. As an example we have that 
$$ \begin{bmatrix} 1 & 1 & ? \cr ? & 1 & 1 \cr 2 & ? & 1 \end{bmatrix} \otimes B := \begin{bmatrix} B & B & ? \cr ? & B & B \cr 2B & ? & B \end{bmatrix}. $$ In our definition of the fractional minimal rank we use $B=I_k$, the $k\times k$ identity matrix. Thus for the partial matrix above we obtain
$$ {\mathcal A} \otimes I_k := \begin{bmatrix} I_k & I_k & ? \cr ? & I_k & I_k \cr 2I_k & ? & I_k \end{bmatrix}. $$ The minimal rank of this partial block matrix (denoted ${\rm mr} ({\mathcal A} \otimes I_k)$), defined to be the lowest possible rank of all its completions, depends on $k$. Recall that a completion of a partial (block) matrix is obtained by specifying the unknown entries (indicated by ?) by appropriately sized matrices over the field ${\mathbb F}$ under consideration. We now define the {\it fractional minimal rank} (denoted ${\rm fmr}$) of a partial matrix ${\mathcal A}$ by
$$ {\rm fmr} ({\mathcal A}):= \inf_{k\in {\mathbb N}} \frac{{\rm mr} ({\mathcal A} \otimes I_k)}{k} .$$
We will see that for block triangular partial matrices the fractional minimal rank and the minimal rank agree. In fact, inspired by the conjecture in \cite{CJRW89}, we conjecture that the fractional minimal rank and the minimal rank agree when the bipartite graph corresponding to the partial matrix is bipartite chordal. As minimal cycles of length 6, 8, 10, etc., are the basic building blocks of bipartite graphs that are not bipartite chordal, it is of interest to determine the fractional minimal rank of corresponding partial matrices. These have the form 
\begin{equation}\label{cyc1}
{\mathcal A} = 
\begin{bmatrix}
a_{11} & a_{12} & ? & \cdots & ?\\
? & a_{22} & a_{23} & \ddots & \vdots  \\
\vdots & \ddots & \ddots&\ddots & ?\\
\\
? & \cdots & ? & a_{n-1,n-1}  & a_{n-1,n} \\
a_{n1} & ? & \cdots & ?  & a_{nn}\\
\end{bmatrix}
\end{equation}
Our main result gives the fractional minimal rank of \eqref{cyc1} (where each $a_{ij}$ is non-zero).  For instance, we find that the fractional minimal rank of $\smat{1&1&?\\?&1&1\\2&?&1} = \frac 32$.
 
The paper is organized as follows. In Section \ref{prelims} we recall some definitions from the literature, introduce the fractional minimal rank, and establish some basic properties. The main result from this section is  that the fractional minimal rank lies between the triangular minimal rank and the minimal rank. In Section \ref{cycle}, we present our main result wherein we compute the fractional minimal rank of the partial matrices corresponding to ``minimal cycles'' (this terminology is explained in Section \ref{open}). In fact, we establish a minimal rank formula for block matrices of the form \eqref{cyc} where the prescribed blocks are invertible. In Section \ref{case zero}, we address the case in which the partial matrix in \eqref{cyc} contains some zero entries; in fact, we show that the triangular and ordinary minimal ranks coincide in all such cases, which means that the fractional minimal rank and ordinary minimal rank coincide. Finally, in Section \ref{open} we discuss this problem within the broader context and suggest some open problems to pursue.

\section{Preliminaries}\label{prelims}

A {\it partial matrix} is an $m\times n$ array $\s A=(a_{ij})_{i=1, j=1}^{m,\ \ n}$ of
elements in a field ${\mathbb F}$,  with some of the entries being specified
(known), and others being unspecified (unknown). In some cases we consider the entries to be matrices of
appropriate size. The unspecified
entries of a partial matrix are typically denoted by $?$, $X$, $x$, $Y$, $y$,
etc. (we use $?$ in this paper).  A {\it completion} of a
partial matrix is a specification of its unspecified entries,
resulting in a conventional (operator) matrix. For instance, 
%$$ 
%\begin{bmatrix} 1 & 2\cr 3 & 4 \end{bmatrix} \ {\rm is \ a \ completion \ of \ the \ partial \ matrix\ } \begin{bmatrix} 1 & ?\cr ? & 4 \end{bmatrix} . 
%$$
$\smat{1&2\\3&4}$ is a completion of the partial matrix $\smat{1&?\\?&4}$.  
%
% Definition of pattern
%

 Given such a partial matrix $\s A$, we define \textit{the pattern of knowns} as the set $J \subset \{1,\dots,m\} \times \{1, \dots,n\}$ such that $(i,j) \in J\iff a_{ij}$ is a known entry. We also write $\s A = \{a_{ij}: (i,j) \in J\}$ to indicate that $J$ is the pattern of knowns for the partial matrix $\s A$.

% %
%
% Ben's Note:
% Is it preferable to have it as it's own line?
%
% %

The {\it minimal
rank} of $\mathcal A$ (notation: mr$({\mathcal A})$) is defined by
\begin{equation} 
{\rm mr}({\mathcal A}) = {\rm min} \big\{{\rm rank} \,B
: B \, {\rm is \,a \,completion \,of} \, {\mathcal A} \big\}.
\end{equation}

% %
%
% Ben's note: we define the term "partial matrix", but not the term "pattern" (as in the pattern of known entries) nor do we define subpattern.
%
% %

We now recall some terminology and results from \cite{W93}: we say that a pattern $T \subseteq \{ 1, \ldots,n \} \times \{1, \ldots, m \}$ is
\emph{triangular} if whenever $(i,j) \in T$
and $(k,l) \in T$, then $(i,l) \in T$ or $(k,j) \in T$. In other
words, a subpattern of the form
%$$\begin{bmatrix} * & ? \\ ? & *
%\end{bmatrix}$$ 
$\smat{*&?\\?&*}$
is not allowed. Denote $\col_j (T) = \{i: (i,j) \in T  \}$. For a pattern
$T\subset \{1,\ldots ,n\}\times \{1,\ldots ,m\}$, one sees that $T$ is triangular if and
only if for all $1\le j_1, j_2\le m$ we have $\col_{j_1} (T)
\subseteq \col_{j_2} (T)$ or $\col_{j_2} (T) \subseteq \col_{j_1}
(T)$. In other words, the pattern $T$ is triangular if and only if
the sets $\col_j (T)$, $1\le j \le m$, may be linearly ordered by
inclusion. This is equivalent to the fact that after suitable
permutations of its rows and columns, the pattern can be
transformed into a block lower triangular one, with the blocks
having appropriate sizes. Clearly, an analogous result holds for
the ``rows'' of a triangular pattern $T$.

% % 
%
% Moved this: theorem should come after we say what a triangular pattern is
%
% %

We say that a partial matrix is triangular if its pattern of knowns is triangular in the above sense.  Triangular patterns are important for our purposes since it is relatively easy to compute the minimal rank of a partial matrix over such a pattern.  In particular, the following formula is given in \cite{W89}.

\begin{thm}\label{2}\cite{W89}
The partial matrix ${\mathcal A} = \{ A_{ij}:1 \leq j \leq i \leq
n \},$ with $A_{ij}$ of size $\nu_i \times \mu_j$, has minimal
rank \begin{equation}\label{minranktriangle} \mr {\mathcal A} = \sum_{i=1}^{n} {\rm rank} \begin{bmatrix} A_{i1}
& \ldots & A_{ii} \cr \vdots & & \vdots \cr A_{n1} & \ldots &
A_{ni}\end{bmatrix}  \quad \quad \quad \quad \end{equation}
$$ \quad \quad \quad \quad \quad \quad
- \sum_{i=1}^{n-1} {\rm rank}
\begin{bmatrix} A_{i+1, 1} & \ldots & A_{i+1, i} \cr \vdots & &
\vdots \cr A_{n1} & \ldots & A_{ni} \end{bmatrix}.$$ 
%In addition,
%if $\nu_1 > 0 $ and $\mu_n > 0$, the minimal rank completion is
%unique if and only if
%$$ {\rm rank} \begin{bmatrix} A_{i1} & \ldots &
%A_{ii} \cr \vdots & & \vdots \cr A_{n1} & \ldots &
%A_{ni}\end{bmatrix} = {\rm rank} A_{n1} , \  i=1,\ldots , n .$$

\end{thm}
%
%The uniqueness statement in Theorem \ref{2} is due to \cite{KW88}.

Next we say that $T$ is a \emph{triangular
subpattern} of $J$ if (1) $T
\subseteq J$ and (2) $T$ is triangular.  The pattern $T$ is called
a \emph{maximal triangular subpattern}
 of $J$ if $T \subseteq S \subseteq J$ with
$S$ a triangular subpattern of $J$ implies that $T=S$.  We now
have the following result.  For $T \subseteq J$ and $\mathcal{A} =
\{ a_{ij} : (i,j) \in J \}$ a partial matrix, let $\mathcal{A}
\mid T$ denote the partial matrix $\mathcal{A} \mid T = \{ a_{ij}:
(i,j) \in T \}$.

Let $J \subseteq \{1, \ldots, n \}
\times \{ 1, \ldots, m \}$ be a pattern, and let $\mathcal{A} = \{
A_{ij}\ : \ (i,j) \in J \}$ be a partial block matrix. Then
\begin{equation}\label{tmr} {\rm mr} \mathcal{A} \geq \max \{ {\rm
mr} (\mathcal{A} \mid T)\ :\  T \subseteq J \ {\rm maximal\
triangular}\}. 
\end{equation} 
The right hand side of \eqref{tmr} is defined to be the {\it triangular 
minimal rank} of $\s A$ (notation: $\tmr ( \s A)$).

We define the tensor product (i.e. Kronecker product) of a partial matrix in a manner corresponding to the tensor product of matrices.  In particular: for an $m \times n$ matrix $A$ and a $p \times q$ matrix $B$, the tensor product $A \otimes B$ is the $(mp)\times (nq)$ matrix defined by 
\[
A \otimes B = 
\mat{a_{11}B & \cdots & a_{1n}B\\
\vdots & \ddots & \vdots\\
a_{m1}B & \cdots & a_{mn}B}
\]
For a partial matrix $\s A$ and matrix $B \in \C^{n \times n}$, we analogously define $\s A \otimes B$ to be a partial matrix consisting of the corresponding formal products, where we define $B\,? = ?_{m \times n}$ for any $?$-entry and any matrix $B$.  So, for example: if we take $\s A = \smat{2&1&1\\1&?&1\\1&1&?}$, we would compute
\[
\s A \otimes I = 
\mat{
2I_3 & I & I\\
I & ?_{3 \times 3} & I\\
I & I & ?_{3 \times 3}
}
\]

We define the {\it direct sum} $\s A \oplus \s B$ of two partial matrices to be the partial matrix
\[
\s A \oplus \s B = \mat{\s A & ? \\ ? & \s B}. 
\]

For a positive integer $b$ (that is, $b \in \N$) we define the {\it b-fold minimal rank} to be
\begin{equation}
\mr_b(\s A) = \mr(\s A \otimes I_b)
\end{equation}
Finally, we define the {\it fractional minimal rank} to be
\begin{equation}\label{fmrinf}
\fmr(\s A) = \inf_{b \in \N} \frac{\mr_b(\s A)}{b}
\end{equation}
Notably, among these the fractional minimal rank is the only one that might not produce an integer value.

%For convenience, $\xmr$ will generically denote any of these $\R$-valued function on partial matrices.  That is, $\xmr$ is a "placeholder" with $\xmr \in \{\rk, \mr,\tmr,\mr_b,\fmr\}$.  We will refer to these as \textit{pseudo-ranks}.

In the next proposition we list some basic properties.
\begin{prop}\label{basics}

For partial matrices $\s A, \s B$ and $\xmr \in \{ \mr,\tmr,\mr_b,\fmr\}$, we have
\begin{itemize}
\item[(i)]
$
\max \{\xmr \s A, \xmr \s B\} \leq 
\xmr\mat{\s A & \s B} \leq \xmr(\s A) + \xmr(\s B) .
$
\item[(ii)]
$
\xmr(\s A \oplus \s B) = \max\{\xmr(\s A), \xmr(\s B)\}
$
\item[(iii)]
If $\s B$ can be attained by permuting the rows/columns of $\s A$ successively, then $\xmr(\s B) = \xmr(\s A)$.
\item[(iv)]
If $J$ is the pattern of knowns for $\s A$ and $T \subseteq J$, then $\xmr(\s A | T ) \leq \xmr (\s A)$.
\end{itemize}
\end{prop}

Item (ii) is proven (for triangular patterns in particular) in \cite{CJRW89}.

% Fmr property
Now, we have an important property of the $b$-fold minimal rank:

\begin{lem}[subadditivity]\label{lem: subadd}
For any partial matrices $\s A \subset \C^{m \times n}$ and any $b,c \in \N$, we have
\[
\mr_{b + c}(\s A) \leq \mr_b(\s A) + \mr_c(\s  A)
\]
\end{lem}

\begin{proof}

Let $\s A = \{a_{ij} : (i,j) \in J\}$.
For $k \in \{b,c\}$, let the completions $A_k \in \s A \otimes I_k$ be such that $\rk(A_k) = \mr_k(\s A)$.  Write 
\[
A_k = 
\mat{
A_k^{(1,1)} & \cdots & A_k^{(1,n)}\\
\vdots & \ddots & \vdots\\
A_k^{(m,1)} & \cdots & A_k^{(m,n)}
}
\]
where each $A_k^{(i,j)} \in \C^{k \times k}$.  In accordance with our definitions of the partial matrices, we note that for $(i,j) \in J$, we have $A_k^{(i,j)} = a_{i,j}I_k$ (whereas the remaining $A_k^{(i,j)}$ are unknown block matrices).

We construct the completion $A_{b + c} \in \s A \otimes I_{b + c}$ as follows: divide $A_{b+c}$ into a block matrix (as we have done with the $A_k$ above), and define
\[
A_{b+c}^{(i,j)} = A_{b}^{(i,j)} \oplus A_{c}^{(i,j)}
\]
Note that for $(i,j) \in J$, $A_{b+c}^{(i,j)} = a_{ij}I_{b+c}$, so that $A_{b + c} \in \s A \otimes I_{b+c}$.  Moreover, we note that there exist permutation matrices $P,Q$ such that
$$
PA_{b+c}Q = \mat{A_b & 0\\0 & A_c}
$$
so that $\rk(A_{b+c}) = \rk(A_b) + \rk(A_c)$.  Thus, we may conclude that
\[
\mr_{b+c}(\s A) \leq \rk(A_{b+c}) = \rk(A_b) + \rk(A_c)
\]
as desired.
\end{proof}

Note that Fekete's lemma \cite{Fekete} applies.

\begin{lem}\cite{Fekete} A sequence $(a_n)_{n \in \N}$ is called subadditive if for all $m,n \in \N$, \\
$a_{m+n} \leq a_m + a_n$.  For any subadditive sequence $(a_n)$, $\lim_{n \to \infty} \frac{a_n}{n}$ exists and is equal to $\inf_{n \in \N} \frac {a_n}{n}$.
\end{lem}

Since the sequence of $b$-fold minimal ranks is subadditive, we may now define the fractional minimal rank as a limit:

\begin{cor}\label{cor: FMR lim}
The limit $\lim_{b \to \infty}\frac {\mr_b(\s A)}{b}$ exists. Moreover,
\[
\lim_{b \to \infty}\frac {\mr_b(\s A)}{b} = 
\fmr(\s A) = \inf_{b \in \N} \frac{\mr_b(\s A)}{b}
\]
\end{cor}

Note that the $\fmr$ of a partial matrix will never be between 0 and 1. In particular, if the partial matrix $\s A$ contains only zero known entries, then we have $\mr_b(\s A) = \mr(\s A) = 0$.  On the other hand, if $\s A$ has a non-zero known entry, then all $b$-fold completions will contain a $b \times b$ multiple of the identity matrix as a submatrix, which means that $\rk(A) \geq b$ for all $b$-fold completions of $\s A$, which means that $\mr_b(\s A) \geq b$.

We also have the following general result.

\begin{thm}\label{t<f<m} For any partial matrix $\s A$, we have
\begin{equation} 
 \tmr (\s A ) \le \fmr (\s A) \le \mr (\s A ). \end{equation}
\end{thm}

\begin{proof}

It is clear that $\fmr(\s A) \leq \mr(\s A)$, since
\begin{equation}
\fmr(\s A) = \inf_{b \in \N} \frac{\mr_b(\s A)}{b} \leq \frac{\mr_1(\s A)}{1} = \mr(\s A)
\end{equation}
On the other hand, we wish to show that $\tmr(\s A) \leq \fmr(\s A)$.  To that end, let $J$ denote the pattern of knowns for $\s A$, and suppose that $T \subseteq J$ is any maximal triangular; it suffices to show that $\mr(\s A \mid T) \leq \fmr(\s A)$.  For convenience, denote $\s T = \s A \mid T$

It follows from theorem $\ref{2}$ that $\mr_b(\s T) = b \mr(\s T)$.  In particular: by part (iii) of proposition \ref{basics} (i.e. by applying a suitable permutation of rows and columns), we may assume without loss of generality that $\s T$ has the form of the matrix presented in theorem $\ref{2}$.  That is, we can say that $\s T = \{T_{ij} : 1 \leq j \leq i \leq n\}$, with $T_{ij}$ of size $\nu_i \times \mu_j$.  With that, we note that $\s T \otimes I = \{T_{ij} \otimes I : 1 \leq j \leq i \leq n\}$.  So, following the formula, we have
\begin{align*} 
\mr_b \s T = 
\mr (\s T \otimes I_b) &= \sum_{i=1}^{n} {\rm rank} \begin{bmatrix} T_{i1} \otimes I_b
& \cdots & T_{ii} \otimes I_b \cr \vdots & & \vdots \cr T_{n1} \otimes I_b & \cdots &
T_{ni} \otimes I_b\end{bmatrix}  \quad \quad \quad \quad 
\\
& \qquad 
- \sum_{i=1}^{n-1} {\rm rank}
\begin{bmatrix} T_{i+1, 1} \otimes I_b & \cdots & T_{i+1, i} \otimes I_b \cr \vdots & &
\vdots \cr T_{n1} \otimes I_b & \ldots & T_{ni} \otimes I_b \end{bmatrix}
\\ & =
{\rm rank} \left[\begin{bmatrix} T_{i1}
& \cdots & T_{ii} \cr \vdots & & \vdots \cr T_{n1} & \cdots &
T_{ni} \end{bmatrix} \otimes I_b \right]  \quad \quad \quad \quad 
\\
& \qquad 
- \sum_{i=1}^{n-1} {\rm rank}
\left[\begin{bmatrix} T_{i+1, 1}& \cdots & T_{i+1, i} \cr \vdots & &
\vdots \cr T_{n1} & \ldots & T_{ni} \end{bmatrix} \otimes I_b \right]
\end{align*}
Now, noting that $\rk(A \otimes B) = \rk(A)\rk(B)$, we may rewrite the above as 
\begin{align*}
\mr(\s T \otimes I_b) &= 
b \cdot \vast[{\rm rank} \begin{bmatrix} T_{i1}
& \cdots & T_{ii} \cr \vdots & & \vdots \cr T_{n1} & \cdots &
T_{ni} \end{bmatrix}   \quad \quad \quad \quad 
\\
& \qquad 
- \sum_{i=1}^{n-1} {\rm rank}
\begin{bmatrix} T_{i+1, 1}& \cdots & T_{i+1, i} \cr \vdots & &
\vdots \cr T_{n1} & \ldots & T_{ni} \end{bmatrix} \vast] = b \cdot \mr (\s T)
\end{align*}
With that, we may finally note that by part (iv) of proposition \ref{basics}, we have
\begin{align*}
\fmr(\s A) &\geq \fmr(\s A \mid T)
= \inf_{b \in \N} \frac{\mr_b(\s A \mid T)}{b} = \mr(\s A \mid T)
\end{align*}
Now, since $\mr(\s A \mid T) \leq \fmr(\s A)$ for every maximal triangular subpattern $T$, we can conclude that $\tmr(\s A) \leq \fmr(\s A)$, as desired.

\end{proof}

An curious property of the fractional minimal rank is that if $\fmr(\s A) < \mr(\s A)$, the minimizer $A \in \s A \otimes I_b$ can not have simultaneously triangularizable block entries. That is, for all partial matrices $\s A$:
 
\begin{prop}
If $\s A \otimes I_b$ achieves its minimal rank with upper triangular blocks, then the minimizing completion $A$ satisfies $\rk(A) \geq b\mr(\s A)$.
\end{prop}

\textit{Proof.} Let $A \in A \otimes I_b$ be given by
\[
A = \mat{
A_{11} & \cdots & A_{1n}\\
\vdots & \ddots & \vdots\\
A_{m1} & \cdots & A_{mn}
}, \qquad A_{i,j} \in \C^{b \times b}
\]
Let $\{e_1,\dots,e_n\}$ denote the canonical basis of $\C^n$.

Suppose that each block $A_{ij}$ is upper triangular.  That is, for fixed $i,j$ and $1 \leq p,q \leq b$, we have $p>q \implies e_p^TA_{ij}e_q = 0$.  In other words, $A$ is such that
\[
p>q \implies (e_i \otimes e_p)^TA(e_j \otimes e_q) = 0
\]
Define the matrix $B$ so that $(e_i \otimes e_p)^TA(e_j \otimes e_q) = (e_p \otimes e_i)^TB(e_q \otimes e_j)$.  In particular, note that $B$ has the form
\[
B = \mat{
B_{11} & \cdots & B_{1b}\\
\vdots & \ddots & \vdots\\
B_{b1} & \cdots & B_{bb}
} \qquad B_{p,q} \in \C^{m \times n}
\]
Now, we note that $B_{p,q} = 0$ whenever $p<q$, so that $B$ is block upper-triangular.  Moreover, we note that if $a_{ij}$ is a known entry of the partial matrix $\s A$, then $A_{ij} = a_{ij}I$, and
\[
e_i^TB_{pp}e_j = e_p^TA_{ij}e_p = a_{ij}
\]
thus, $B_{pp} \in \s A$.  We can conclude that
\[
\rk A = \rk B \geq \sum_{p=1}^b \rk B_{pp} \geq b \mr(\s A)
\]
which was the desired result.

\hfill $\square$

As an aside we observe that two $k\times k$ matrices are simultaneously upper triangularizable if and only if $w(A,B)(AB-BA)$ has trace 0 for all words $w(A,B)$ in $A$ and $B$ of length at most $k^2-1$; see \cite{AK}.  It is also notable that if a collection of matrices over an algebraically closed field commute, then they are simultaneously triangularizable.

\vspace{.5 cm}

%
% NEED TO DO SOMETHING WITH THE
% BIPARTITE GRAPH STUFF
%

\section{The minimal cycle case}\label{cycle}
The main result in this section is the following. 

\begin{thm}\label{cyclethm} Let $\F$ be a field (not necessarily algebraically closed). Consider the partial matrix $\s A$ over $\F$ given by
\begin{equation}\label{cyc}
{\mathcal A} = 
\begin{bmatrix}
a_{11} & a_{12} & ? & \cdots & ?\\
? & a_{22} & a_{23} & \ddots & \vdots  \\
\vdots & \ddots & \ddots&\ddots & ?\\
\\
? & \cdots & ? & a_{n-1,n-1}  & a_{n-1,n} \\
a_{n1} & ? & \cdots & ?  & a_{nn}\\
\end{bmatrix}. 
\end{equation}
In case all specified $a_{ij} \in \F$ are nonzero and
$$  a_{12}a_{23}\cdots a_{n-1,n}a_{n1}\neq a_{11}a_{22}\cdots a_{n-1,n-1}a_{nn} ,$$ we have that 
$$1=\tmr (\s A) < \fmr (\s A ) = \frac{n}{n-1} < \mr (\s A ) = 2 . $$
Otherwise $$ \tmr (\s A) = \fmr (\s A ) =  \mr (\s A ) = 1 . $$
\end{thm}

We will first prove the following, which generalizes a result by \cite{Tian}, who considered the case $n=3$. 

\begin{thm}\label{tiangen} Consider the block partial matrix over $\F$
\[
\s A = 
\mat{
A_{11} & A_{12} & ? & \cdots & ?\\
? & A_{22} & A_{23} & \ddots & \vdots  \\
\vdots & \ddots & \ddots&\ddots & ?\\
\\
? & \cdots & ? & A_{(n-1)(n-1)}  & A_{(n-1)n} \\
A_{n1} & ? & \cdots & ?  & A_{nn}\\
},
\]
where each $A_{ij}$ is an invertible matrix of size $k$. Put
$$ H = A_{11}^{-1}A_{12}A_{22}^{-1}A_{23}A_{33}^{-1} \cdots A_{n-1,n-1}^{-1}A_{n-1,n}A_{nn}^{-1}A_{n1} $$ and $d = \dim \ker(H - I)$.
Then we have 
\[
\mr(\s A) \geq \frac{nk - d}{n-1}. 
\]
Moreover: when $d = k$, we also have equality; i.e. $\mr(\s A) = k$.

Finally, when $k=n-1$, $d=0$, and $H$ has all $k$ of its eigenvalues in $\F$, we have equality: $\mr(\s A)=n$.  (Notably, if $\F$ is algebraically closed, then $H$ necessarily has all of its eigenvalues in $\F$).

\end{thm}

For a subspace $U \subset \C^k$, we define the \emph{codimension} of $U$ by $\codim(U) = \dim(\C^k/U) = k - \dim(U)$.  We note the following:

\begin{lem}\label{lem:cd}
For spaces $U, V \subset \C^k$, we have $\codim(U \cap V) \leq \codim(U) + \codim(V)$.
\end{lem}

\begin{proof} It suffices to note that the linear map from $\C^k/(U \cap V)$ to $(\C^k/U) \oplus (\C^k/V)$ given by $x + (U \cap V) \mapsto (x+U, x+V)$ is well-defined and injective. \end{proof} 

\begin{lem} \label{lem:sim} If $A$ is invertible, then $AB$ is similar to $BA$
\end{lem}

\begin{proof} $A(BA)A^{-1} = AB$. \end{proof} 

\noindent 
{\it Proof of Theorem \ref{tiangen}.} 
To begin, we note that for any $A \in \s A$: if we take
\begin{gather*}
P = \mat{
A_{11}^{-1} \\
&  A_{11}^{-1}A_{12}A_{22}^{-1}\\
&& A_{11}^{-1}A_{12}A_{22}^{-1}A_{23}A_{33}^{-1}\\
&&& \ddots 
},\\
Q = 
\mat{
I \\
& A_{12}^{-1}A_{11}\\
&& A_{23}^{-1}A_{22}A_{12}^{-1}A_{11}
\\
&&& \ddots 
}
\end{gather*}
then the matrix $PAQ$ is a rank-equivalent element of the partial matrix
\[
\s A_H = 
\mat{
I & I & ? & \cdots & ?\\
? & I & I & \ddots & \vdots  \\
\vdots & \ddots & \ddots&\ddots & ?\\
? & \cdots & ? & I  & I \\
H & ? & \cdots & ?  & I\\
}
\]
For the \textbf{invertible matrix} $H = A_{11}^{-1}A_{12}A_{22}^{-1}A_{23}A_{33}^{-1} \cdots A_{n-1,n-1}^{-1}A_{n-1,n}A_{nn}^{-1}A_{n1}$. Thus, it suffices in general to consider partial matrices of this latter form; that is, we assume without loss of generality that $\s A = \s A_H$.

Now, we show that $\mr(\s A) \geq \frac {nk - d}{n-1}$.

Let $A =  [A_{ij}]_{1 \leq i,j \leq n}$ be a completion of $\s A \otimes I_k$ (so that $A_{ij} \in \C^{k \times k}$).  That is, we have
\begin{align*}
A &= 
\mat{
A_{11} & A_{12} & A_{13} & \cdots & A_{1n}\\
A_{21} & A_{22} & A_{23} & \ddots & \vdots  \\
\vdots & \ddots & \ddots&\ddots & A_{(n-2)n}\\
A_{(n-1)1} & \cdots & A_{(n-1)(n-2)} & A_{(n-1)(n-1)}  & A_{(n-1)n} \\
A_{n1} & A_{n2} & \cdots & A_{n(n-1)}  & A_{nn} \\
} 
\\ & = 
\mat{
I & I & A_{13} & \cdots & A_{1n}\\
A_{21} & I & I & \ddots & \vdots  \\
\vdots & \ddots & \ddots&\ddots & A_{(n-2)n}\\
A_{(n-1)1} & \cdots & A_{(n-1)(n-2)} & I  & I \\
H & A_{n2} & \cdots & A_{n(n-1)}  & I\\
}
\end{align*}
We wish to show that $\mr(\s A) \geq \frac {nk - d}{n-1} $.  That is: we suppose that $A$ has rank $r$, and we want to show that $r \geq \frac {nk - d}{n-1}$.

In the following, let $E_{\lambda}(M) = E^R_{\lambda}(M) = \ker(\lambda I - M)$ denote the eigenspace of $M$ associated with eigenvalue $\lambda$.  Let $E^L_{\lambda}(M) = \ker(\lambda I - M^T)$ denote the ``left-eigenspace'' of $M$ associated with the eigenvalue $\lambda$.

We note that the matrix
\begin{align*} 
\mat{I & -H^{-1}\\0 & I}\mat{A_{11} & A_{1n}\\A_{n1} & A_{nn}} &=
\mat{I & -H^{-1}\\0 & I}\mat{I & A_{1n} \\H & I}
\\ & =
\mat{0 & A_{1n} - H^{-1} \\ H & I}
\end{align*}
has rank at most $r$.  Noting that the bottom-left corner has full rank, we conclude that 
\[
\rk(A_{1n} - H^{-1}) \leq r - k
\]  
So, we can say
\begin{align*} 
\codim(E_1(HA_{1n})) = \codim(\ker(HA_{1n} - I)) =\codim(\ker(A_{1n} - H^{-1})) \leq r - k
\end{align*}

Now, for any $1<i<j<n$ with $j > i + 1$, we can consider the matrix
\begin{align*}
    \mat{I & 0\\-A_{(j-1),(i+1)} & I} 
    \mat{A_{i,(i+1)} & A_{i,j}\\A_{(j-1),(i+1)} & A_{(j-1),j}} &=
    \mat{I & 0\\-A_{(j-1),(i+1)} & I} 
    \mat{I & A_{i,j}\\A_{(j-1),(i+1)} & I}
    \\ & = 
    \mat{I & A_{i,j}\\
    0 & I - A_{(j-1),(i+1)}A_{i,j}}
\end{align*}
which has rank at most $r$.  As before, we note that the upper-left corner has full rank to conclude that
\begin{equation}\label{eq:cd1}
\codim(E_1(A_{j-1,i+1}A_{ij})) = \rk(I - A_{(j-1),(i+1)}A_{i,j}) \leq r - k
\end{equation}

That is, $\codim(E_1(A_{(j-1),(i+1)}A_{i,j})) \leq r - k$.

\begin{claim}
 $E_{1}(HA_{i,j}) \cap E_1(A_{(j-1),(i+1)}A_{i,j}) \subset E_{1}(A_{(j-1),(i+1)}H^{-1})$
 \end{claim}

\noindent \textit{Proof.} For any $x \in E_{1}(HA_{i,j}) \cap E_1(A_{(j-1),(i+1)}A_{i,j})$, we have
\[
A_{(j-1),(i+1)}H^{-1}x = A_{(j-1),(i+1)}H^{-1}[HA_{i,j}x] = A_{(j-1),(i+1)}A_{i,j} x = x \qquad \square
\]
Consequently, we can conclude that
\begin{align*}
\codim(E_{1}(A_{(j-1),(i+1)}H^{-1})) &\leq 
\codim(E_{1}(HA_{i,j}) \cap E_{1}(A_{i,j}A_{(j-1),(i+1)}))
\\ & \leq 
\codim(E_{1}(HA_{i,j})) + \codim(E_{1}(A_{i,j}A_{(j-1),(i+1)})) 
\\ & \overset{\ref{eq:cd1}}{\leq}
\codim(E_{1}(HA_{i,j})) + (r - k)
\end{align*}

Similarly: for any $1 < j < i < n-1$, we can consider the matrix
\begin{align*}
\mat{I & 0\\ -A_{i,j} & I}
\mat{A_{j,j} & A_{j,i}\\ A_{i,j} & A_{i,i}} &= 
\mat{I & 0\\ -A_{i,j} & I}
\mat{I & A_{i,j}\\ A_{i,j} & I}
\\ & = 
\mat{I & A_{j,i}\\0 & I - A_{i,j}A_{j,i}}
\end{align*}
has rank at most $r$.  Noting that the upper-left block has full rank, we conclude that 
\begin{equation}\label{eq:cd2}
\codim(E_1^L(A_{i,j}A_{j,i})) = \rk (I - A_{i,j}A_{j,i}) \leq r - k
\end{equation}
That is, $\codim[E_1(A_{i,j}A_{j,i})] \leq r - k$.

\begin{claim}\label{bigclaim}
$[E_1^L(A_{i,j}A_{j,i}) \cap E_{1}^L(A_{i,j}H^{-1}) \subset E_{1}^L(HA_{j,i})$
\end{claim}

\noindent \textit{Proof of claim \ref{bigclaim}.} For any $x \in [E_1^L(A_{i,j}A_{j,i}) \cap E_{1}^L(A_{i,j}H^{-1})$, we have
\[
x^T HA_{j,i} = [x^T A_{i,j}H^{-1}]HA_{j,i} = x^T A_{i,j} A_{j,i} = x^T \qquad \square
\]
Consequently, we can conclude that 
\begin{align*} 
\codim(E_{1}^L(HA_{j,i}) &\leq 
\codim(E_1^L(A_{i,j}H^{-1}) \cap E_1^L(A_{i,j}A_{j,i}))
\\ & \leq 
\codim(E_1^L(A_{i,j}H^{-1})) +  \codim(E_1^L(A_{i,j}A_{j,i}))
\\ & \overset{\eqref{eq:cd2}}{\leq}
\codim(E_1^L(A_{i,j}H^{-1})) +  (r - k)
\end{align*}
So, we now have the following two inequalities:
\begin{align}
    \codim(E_{1}(A_{(j-1),(i+1)}H^{-1})) & \leq \codim(E_{1}(HA_{i,j})) + (r-k) \qquad j > i + 1
    \label{eq:cd3}
    \\
    \codim(E_{1}(HA_{j,i})) & \leq \codim E_{1}(A_{i,j}H^{-1}) + (r-k) \qquad j < i < n
    \label{eq:cd4}
\end{align}

\textbf{Case 1: If $n$ is odd,} let $p = (n+3)/2$ and $q = (n - 1)/2 = p - 2$.  We have
\begin{align*}
    \codim[E_{1}(HA_{p,q})] 
    & \overset{\eqref{eq:cd4}}{\leq} \codim[E_{1}(A_{q,p}H^{-1})] + (r - k)
    \\ & \overset{\eqref{eq:cd3}}{\leq} \codim[E_{1}(HA_{p+1,q-1})] + 2(r - k)
    \\ & \overset{\eqref{eq:cd4}}{\leq} \codim[E_{1}(A_{q-1,p+1}H^{-1})] + 3(r - k)
    \\ & \qquad \vdots
    \\ & \overset{\eqref{eq:cd3}}{\leq} \codim[E_{1}(HA_{1,n})] + (n-3)(r-k) = (n-2)(r-k)
\end{align*}
However, we note that the matrix 
\begin{align*} 
\mat{I&-I\\0&I} \mat{A_{p,(q-1)} & A_{p,q}\\ A_{(p+1),(q-1)} & A_{(p+1),q}} &=
\mat{I&-I\\0&I} \mat{I & A_{p,q}\\ I & I}
\\ & =
\mat{0 & A_{p,q} - I\\I & I}
\end{align*}
has rank at most $r$, which tells us that $\rk(A_{p,q} - I)$ is at most $r - k$.  That is, $\codim(E_1(A_{p,q})) \leq r - k$.  

We note that $E_1(A_{p,q}) \cap E_1(HA_{p,q}) \subset E_1(H)$, since for any $x \in E_1(A_{p,q}) \cap E_1(HA_{p,q})$, we have
\[
Hx = HA_{p,q}x = x \qquad
\]

We can then say that $\codim[E_1(A_{p,q}) \cap E_1(HA_{p,q})] \geq \codim(E_1(H)) = k-d$.  It follows that 
\begin{align*} 
(k-d) &\leq \codim[E_{1}(HA_{p,q}) \cap E_1(A_{p,q})] 
\\ & \leq \codim(E_{1}(HA_{p,q})) + \codim(E_{1}(A_{p,q}))
\\ & \leq (r - k) + (n-2)(r - k) = (n-1)(r - k)
\end{align*}
Thus, we have
\[
(k-d) \leq (n-1)r - (n-1)k \implies \frac {nk - d}{n-1} \leq r
\]
which was what we wanted to prove.

\textbf{Case 2: If $n$ is even,} let $p = 1 + n/2$ and $q = n/2 = p-1$.  We have
\begin{align*}
    \codim[E_{1}(A_{p,q}H^{-1})] 
    & \overset{\eqref{eq:cd3}}{\leq} \codim[E_{1}(HA_{q+1,p-1})] + (r - k)
    \\ & \overset{\eqref{eq:cd4}}{\leq} \codim[E_{1}(A_{p-1,q+1}H^{-1})] + 2(r - k)
    \\ & \overset{\eqref{eq:cd3}}{\leq} \codim[E_{1}(HA_{q+2,p-2})] + 3(r - k)
    \\ & \qquad \vdots
    \\ & \overset{\eqref{eq:cd3}}{\leq} \codim[E_{1}(HA_{1,n})] + (n-3)(r-k) = (n-2)(r-k)
\end{align*}
However, we note that the matrix 
\begin{align*} 
\mat{I&0\\-I&I} \mat{A_{(p-1),q} & A_{(p-1),(q+1)}\\ A_{p,q} & A_{p,(q+1)}} &=
\mat{I&0\\-I&I} \mat{I & I\\ A_{p,q} & I}
\\ & =
\mat{I & I\\A_{p,q} - I & 0}
\end{align*}
has rank at most $r$, which tells us that $\rk(A_{p,q} - I)$ is at most $r - k$.  That is, $\codim(E_1(A_{p,q})) \leq r - k$.  

We note that $E_1^L(A_{p,q}) \cap E_1^L(A_{p,q}H^{-1}) \subset E_1^L(H^{-1})$, since for any $x \in E_1^L(A_{p,q}) \cap E_1^L(HA_{p,q})$, we have
\[
x^TH^{-1} = x^TA_{p,q}H^{-1} = x^T \qquad \square
\]
So, we have $E_1^L(A_{p,q}) \cap E_1^L(A_{p,q}H^{-1}) \subset E_1^L(H^{-1}) = E_1^L(H)$, which is to say that $\codim[E_1(A_{p,q}) \cap E_1(A_{p,q}H^{-1})] \geq k - d$.  It follows that 
\begin{align*} 
k - d &= \codim[E_{1}^L(A_{p,q}H^{-1}) \cap E_1(A_{p,q})] 
\\ & \leq \codim(E_{1}(A_{p,q}H^{-1})) + \codim(E_{1}(A_{p,q}))
\\ & \leq (r - k) + (n-2)(r - k) = (n-1)(r - k)
\end{align*}
Thus, we have
\[
k -d \leq (n-1)r - (n-1)k \implies \frac {nk - d}{n-1} \leq r
\]
which was the desired inequality.

To show that equality holds in the case of $k = n-1$ and $d = 0$,  It suffices to construct a completion $A \in \s A$ with $\rk(A) \leq n$.

First, we consider the case in which $H$ is in its Jordan canonical form.  So, we have
\[
H = 
\mat{
\lambda_1 & \eps_1\\
& \lambda_2 & \eps_2\\
&&\ddots & \ddots \\
&&&\lambda_{n-2} & \eps_{n-2}\\
&&&&\lambda_{n-1}
}
\]
With $\eps_j \in \{0,1\}$ for $j = 1,\dots,n-2$ and $\lambda_j \notin \{0,1\}$ for $j = 1, \dots, n-1$.  We construct the matrix $A$ as follows: take

Take

\begin{align*}
A &= 
\mat{
A_{11} & A_{12} & A_{13} & \cdots & A_{1n}\\
A_{21} & A_{22} & A_{23} & \ddots & \vdots  \\
\vdots & \ddots & \ddots&\ddots & A_{(n-2)n}\\
A_{(n-1)1} & \cdots & A_{(n-1)(n-2)} & A_{(n-1)(n-1)}  & A_{(n-1)n} \\
A_{n1} & A_{n2} & \cdots & A_{n(n-1)}  & A_{nn} \\
} 
\\ & = 
\mat{
I & I & A_{13} & \cdots & A_{1n}\\
A_{21} & I & I & \ddots & \vdots  \\
\vdots & \ddots & \ddots&\ddots & A_{(n-2)n}\\
A_{(n-1)1} & \cdots & A_{(n-1)(n-2)} & I  & I \\
H & A_{n2} & \cdots & A_{n(n-1)}  & I\\
}
\end{align*}

We fill in the first $n-1$ columns by taking $A_{11} = I$ and
\[
A_{i1} = 
\mat{
\lambda_1 & \eps_1\\
&\lambda_2 & \eps_2\\
&&\ddots & \ddots \\
&&&&\eps_{i-2}\\
&&&&\lambda_{i-2}\\
&&&&&I_{n-i-1}
} \quad i = 2,\dots,n
\]
We then define the $n$th column of $A$ by 
\[
\operatorname{col}_1(A_{12}) = 
\operatorname{col}_1(A_{22}) = (1,0,\dots,0)^T
\]
\[
\operatorname{col}_1(A_{i2}) = (1 - \eps_1, 1 - \lambda_2 - \eps_2, \dots, 1 - \lambda_{i-1} - \eps_{i-1}, 1 - \lambda_i,0,\dots,0)^T \quad i = 3, \dots, n
\]
We define the rest of the matrix $A$ by taking each remaining column to be appropriate linear combinations of previous columns.  

For $q > j - 1$, we can take
\[
\operatorname{col}_q(A_{ij}) = \operatorname{col}_q(A_{i1})
\]
so that $\operatorname{col}_{nj + q}(A) = \operatorname{col}_q(A)$.  For $q < j - 1$, we can take
\[
\operatorname{col}_1(A_{ij}) = \frac 1{\lambda_1} \col_1(A_{i1})
\]
\[
\col_q(A_{ij}) = -\eps_q \col_{q-1}(A_{ij}) + \frac 1{\lambda_q}\col_q(A_{i1}) \quad q = 2,\dots,j
\]
Finally, in the case that $q = j + 1$, we can take
\[
\col_{j+1}(A_{ij}) = \col_1(A_{i2}) + \sum_{q=1}^{j-1} \col_q(A_{i1}) - \sum_{q = 1}^{j-2} \col_q(A_{i1})
\]
With the above construction, we see that for every block-row $i$ and block-column $j$, there exist coefficients $\alpha_{1,j,q},\dots,\alpha_{n,j,q}$ (independent of $i$) such that 
\[
\col_q[A_{ij}] = \sum_{k=1}^{n-1} \alpha_{k,j,q} \col_k(A_{i1}) + \alpha_{n,j,q} \col_1(A_{i2})
\]
which in other words, means that 
\[
\col_{n(j-1) + q}(A) = \sum_{k=1}^n \alpha_{n,j,q} \col_k(A)
\]
That is, each column of $A$ is a linear combination of the first $n$ columns, which implies that $\rk(A) \leq n$, as desired.  It is easy to verify that the $A_{ij}$ as constructed above is a completion of the matrix pattern $\s A$.

% maybe include \textit{some kind} of computation in order to demonstrate that the resulting pattern is indeed an element of $\s A_H$.  I'm not sure how to do so in a clear way.

Now, if $H$ has all its eigenvalues in $\F$, then there exists an invertible matrix $S$ such that $H = SJS^{-1}$, and $J$ is in Jordan canonical form.  Suppose that $A \in \s A_J$ is the matrix given by the above construction (and so has rank at most $n$).  Then take
\[
\tilde A = 
[I \otimes S] A [I \otimes S^{-1}]
\]
We find that 
\[
\tilde A_{n1} = SA_{n1}S^{-1} = SJS^{-1} = H
\]
And for the other pairs $i,j$ such that $i = j$ or $i = j+1$, we have
\[
\tilde A_{ij} = SA_{ij}S^{-1} = SIS^{-1} = I
\]
So, we see that $\tilde A \in \s A_H$, and that $\rk(\tilde A) \leq n$.  Thus, we have reached the desired conclusion.

Equality in the case of $d = k$ is trivial: if $d = K$, then $H = I$.  The matrix
\[
\mat{I & I & \cdots & I\\I & I & \cdots & I \\ \vdots & \vdots & \ddots & \vdots \\ I & I & \cdots & I} = 
\mat{1 & 1 & \cdots & 1\\1 & 1 & \cdots & 1 \\ \vdots & \vdots & \ddots & \vdots \\ 1 & 1 & \cdots & 1}\otimes I_k
\]
is a completion of $\s A_H$ with rank $k$.

\hfill $\square$

We may now prove our main result:

\noindent 
{\it Proof of Theorem \ref{cyclethm}.}

Let $k \in \N$ be arbitrary,  Noting that $\s A \otimes I_k$ is of the form described in theorem \ref{tiangen}, we compute $H = hI_k$, where
\[
h = \frac {a_{12}a_{23}\cdots a_{n-1,n}a_{n1}}{a_{11}a_{22}\cdots a_{n-1,n-1}a_{nn} }
\]
To begin, we consider the case in which $h = 1$.  By theorem \ref{tiangen}, we note that $d = \dim \ker(H - I) = k$ and compute
\[
\fmr(\s A) = \lim_{k \to \infty} \frac{\mr(\s A \otimes I_k)}{k} = \lim_{k \to \infty} \frac{k}{k} = 1
\]
In the case of $h \neq 1$, we note that $d = \dim \ker (H - I) = 0$.  By the inequality, we see that 
\[
\fmr(\s A) = \inf_{k \in \N} \frac{\mr(\s A \otimes I_k)}{k} \geq \inf_{k \in \N} \frac{1}{k} \cdot \frac{nk - k}{n-1} = \frac{n}{n-1} 
\]
On the other hand, by taking $k = (n-1)$ (and applying the theorem) we see that
\[
\fmr(\s A) = \inf_{k \in \N} \frac{\mr(\s A \otimes I_k)}{k} \leq \frac{\mr(\s A \otimes I_n)}{n-1} = \frac n{n-1}
\]
So that $\fmr(\s A) = \frac{n}{n-1}$, as desired.

\hfill $\square$

\section{On the case of $a_{ij} = 0$}\label{case zero}

% use \cite{W93}

If we consider the partial matrix as given in (\ref{cyc}) with $a_{ij} = 0$ for any $i,j$, we find that that there is nothing interesting to say about the fmr.  In particular:

\begin{prop}
Let $\s A$ be given as in \ref{cyc} with $a_{ij} \in \F$, and suppose that at least one of the given entries $a_{ij}$ is zero.  In all such cases, we have
\[
\mr(\s A) = \fmr(\s A) = \tmr(\s A) \leq 2
\]
\end{prop}

\textit{Proof:} We note that there exist invertible diagonal matrices $P,Q$ such that for any $A \in \s A$ the matrix $PAQ$ is a rank-equivalent element of the partial matrix

%\textit{Proof:} First, note that with a suitable permutation of the rows and columns, we can move the zero-entry $a_{ij}$ to the $1,n$ spot, while the $?$ entries remain where they are.  That is, we may assume without loss of generality that $a_{1n} = 0$.  

\begin{equation}
\s A_0 = 
\begin{bmatrix}
b_{11} & b_{12} & ? & \cdots & ?\\
? & b_{22} & b_{23} & \ddots & \vdots  \\
\vdots & \ddots & \ddots&\ddots & ?\\
\\
? & \cdots & ? & b_{n-1,n-1}  & b_{n-1,n} \\
b_{1n} & ? & \cdots & ?  & b_{nn}\\
\end{bmatrix}
\end{equation}
Where $b_{ij} = 1$ if $a_{ij} \neq 0$ and otherwise, $b_{ij} = 0$.  That is: without loss of generality, we suppose that the entries $a_{ij}$ were taken from $\{0,1\}$, which is to say that $\s A = \s A_0$.  So, from this point forward, we suppose that $a_{ij} \in \{0,1\}$.
%
%Let $J = \{(i,i): 1 \leq i \leq n\} \cup \{(i,i+1)\ : 1 \leq i \leq n\} \cup \{(1,n)\}$ denote the pattern of knowns.

\vspace{1 cm}

\noindent \textbf{Case 1:} $n$ is even

First, we show that $\mr(\s A) \leq 2$.  To see this, we note that for $1 \leq j \leq n-2$, the known rows of the $j$th column are disjoint from the known rows of the $(j+2)$th column.  With that in mind, we define
\begin{gather}
x_1 = (a_{11},a_{23},a_{33},a_{45},a_{55},\dots,a_{n-2,n-1},a_{n-1,n-1},0)^T\\
x_2 = 
(a_{12},a_{22},a_{34},a_{44},\dots,a_{n,n-1},a_{nn})^T
\end{gather}
We find that the completion
\begin{equation}
A = \mat{x_1 & x_2 & x_1 & x_2 & \cdots & x_1 & x_2}
\end{equation}
satisfies $\rk(A) \leq 2$, so that $\mr(\s A) \leq 2$.

Now, we wish to show that $\tmr(\s A) = \mr(\s A)$. First, consider the case in which the known entries of a row/column are both $0$.  In this scenario, we can remove the zero row/column without affecting the minimal rank, and what we're left with is a banded pattern.  Following Theorem 1.1 of \cite{W93}, we see that the $\tmr(\s A)$ and $\mr(\s A)$ are equal, as desired.

For instance, if $a_{11} = a_{1n} = 0$, we may remove the first column, which is to say that we note
\begin{align*}
\xmr \begin{bmatrix}
0 & a_{12} & ? & \cdots & ?\\
? & a_{22} & a_{23} & \ddots & \vdots  \\
\vdots & \ddots & \ddots&\ddots & ?\\
\\
? & \cdots & ? & a_{n-1,n-1}  & a_{n-1,n} \\
0 & ? & \cdots & ?  & a_{n,n}\\
\end{bmatrix} 
 & = 
\xmr 
\begin{bmatrix}
a_{12} & ? & \cdots & ?\\
a_{22} & a_{23} & \ddots & \vdots  \\
? & \ddots&\ddots & ?\\
\\
\vdots & \ddots & a_{n-1,n-1}  & a_{n-1,n} \\
? & \cdots & ?  & a_{n,n}\\
\end{bmatrix}
\end{align*}
and we can see that what remains is a banded pattern.

Now, we note that if $\smat{1&?\\0&1}$ (or its row/column permutations) appears as a submatrix in $\s A$, then we have $\tmr(\s A) \geq \tmr \smat{1&?\\0&1} = 2$, using part (i) of proposition \ref{basics}.  Noting that $\tmr(\s A) \leq \mr(\s A)$, we may conclude that $\tmr(\s A) = \mr (A) = 2$.

If this submatrix does \textbf{not} appear in the pattern, then we must have a zero row/column.  Thus, we can remove the zero row/column without affecting the minimal rank, and what we're left with is a banded pattern.  Following Theorem 1.1 of \cite{W93}, we see that the $\tmr(\s A)$ and $\mr(\s A)$ are equal, as desired.

\vspace{1 cm}

\noindent \textbf{Case 2:} $n$ is odd.

Suppose that $\s A$ has a row/column with zeros as known entries.  Then, we can remove the zero row/column without affecting the minimal rank, and what we're left with is a banded pattern.  Following Theorem 1.1 of \cite{W93}, we see that the $\tmr(\s A)$ and $\mr(\s A)$ are equal.

We can see moreover that the tmr of the resulting banded pattern is at most $2$.  To see that this is the case, note that the resulting banded pattern will (up to a permutation of rows and columns) have the form
\begin{equation}
\s B = 
\begin{bmatrix}
a_{12} & ? & \cdots & ?\\
a_{22} & a_{23} & \ddots & \vdots  \\
? & \ddots&\ddots & ?\\
\\
\vdots & \ddots & a_{n-1,n-1}  & a_{n-1,n} \\
? & \cdots & ?  & a_{n,n}\\
\end{bmatrix}
\end{equation}
We then note that any triangular subpattern of $\s B$ will contain at most three known entries (that is, any subpattern with more than three knowns will necessarily contain the forbidden subpattern $\smat{x&?\\?&y}$).  That is, a maximal triangular subpattern will have the form (up to some permutation or transpose)
\begin{equation}
\s B_T = 
\begin{bmatrix}
a_{12} & ? & \cdots & ?\\
a_{22} & a_{23} & \ddots & \vdots  \\
? & \ddots&\ddots & ?\\
\\
\vdots & \ddots & ?  & ? \\
? & \cdots & ?  & ?\\
\end{bmatrix}
\end{equation}

Thus, we have $\mr(\s A) = \mr(\s B) \mr(\s B_T) \in \{0,1,2\}$.

Now, suppose that $\s A$ has a row/column with two \textit{non-zero} known entries.  That is, up to a suitable permutation of the rows and columns, our pattern has one of the following form:
\begin{equation}\label{1pat1} 
\s A = 
\mat{
1 & 1 & ? & \cdots & ?\\
? & a_{22} & a_{23} & \ddots & \vdots  \\
\vdots & \ddots & \ddots&\ddots & ?\\
\\
? & \cdots & ? & a_{n-1,n-1}  & a_{n-1,n} \\
a_{1n} & ? & \cdots & ?  & a_{n,n}\\
}
\end{equation}
%Where we no longer presume that $a_{1n} = 0$.  
We may show that $\mr(\s A) \leq 2$ in a manner similar to the even case.   As in the even case, we note that for $1 \leq j \leq n-2$, the known rows of the $j$th column are disjoint from the known rows of the $(j+2)$th column.  With that in mind, we define
\begin{gather}
x_1 = 
(1,a_{22},a_{34},a_{44},\dots,a_{n-2,n-1},a_{n-1,n-1},a_{1,n})^T\\
x_2 = (0,a_{23},a_{33},a_{45},a_{55},\dots,a_{n-1,n},a_{n,n})^T
\end{gather}
We see that the completion
\begin{equation}
A = \mat{x_1&x_1&x_2&x_1&x_2 & \cdots & x_1 & x_2}
\end{equation}
satisfies $\rk(A) \leq 2$.  From there, we again proceed as in the even case: if $\smat{1&?\\0&1}$ appears as a subpattern, then we have $\tmr(\s A) \geq 2$, which leads to the desired conclusion.  If not, then a zero-row must occur, which brings us to a banded pattern, and from there we know that $\tmr = \mr$ by \cite{W93}. 

The only case left to consider is that in which each row and column contains both a zero and non-zero known entry.  Up to a suitable permuation of the rows and columns, we may suppose that $\s A$ has the form
\begin{equation}
\s A = 
\mat{
1&0&? & \cdots & ?\\
? & 1 & 0 & \ddots & \vdots\\
\vdots & \ddots & \ddots & \ddots & ?\\
\\
? & \cdots & ? & 1 & 0\\
0 & ? & \cdots & ? & 1
}
\end{equation}

In this case, we define
\begin{gather}
x_1 = 
(0,1,0,1,\dots,0,1,0,1,1)^T\\
x_2 = (1,0,1,0\dots,1,0,1)^T
\end{gather}

We see that the completion
\begin{equation}
\mat{x_2 - x_1 & x_2 & x_1 & x_2 & \cdots & x_2 & x_1}
\end{equation}
has rank $2$.  Moreover, the upper-left $2 \times 2$ submatrix of $\s A$ is $\mat{1&0\\?&1}$, which has $\tmr$ 2. So, $\mr(\s A) = \tmr(\s A) = 2$. \hfill $\square$

\section{Thoughts on the problem}\label{open}

\subsection{Motivation for our case}

We may associate with any pattern $J$ a bipartite graph.  A {\it bipartite graph} is an undirected
graph $G=(V,E)$ in which the vertex set can be partitioned as
$V=X\cup Y$ such that each edge in $E$ has one endpoint in $X$ and
one in $Y$. We denote such bipartite graphs by $G=(X,Y,E)$. With the specified/unspecified patterns of a partial matrix
$(A_{ij})_{i,j=1}^{m,n}$, we associate the bipartite graph
$G=(X,Y,E)$, where $X=\{1,\ldots ,m\}$, $Y=\{1,\ldots ,n\}$, and
$(u,v)\in E$ if and only if $u\in X$, $v\in Y$, and $A_{u,v}$ is
specified. In any graph, we define a \textit{cycle} to be a tuple $(v_1,\dots,v_m)$ of distinct vertices such that $(v_i,v_{i+1}) \in E$ for $i = 1,\dots,m-1$ and $(v_{m},v_1) \in E$.  In such a cycle, any other edge of the form $(v_i,v_j) \in E$ (i.e. any edge between non-adjacent vertices of the cycle) is called a \textit{chord}, and a cycle without a chord is a \textit{minimal cycle}. A bipartite graph is called {\it bipartite chordal} if it does not contain minimal cycles of length 6 or higher.

If $J$ is an $m \times n$, then the associated bipartite graph is $G_J = (\{1,\dots,m\},\{1,\dots,n\},E)$ where $ij \in E$ if and only if $(i,j) \in J$. In \cite{CJRW89}, it is shown that if $J$ is a pattern whose associated graph is not bipartite chordal, then there exists a partial matrix over this pattern for which $\tmr(\s A) < \mr(\s A)$.  Moreover, it is conjectured that when $\s A$ is a graph whose associated graph \textbf{is} bipartite chordal, we have $\tmr(\s A) = \mr(\s A)$.

With this conjecture in mind, it is interesting to consider the fractional minimal rank of a partial matrix whose pattern is a cycle (of arbitrary size) since in these cases where the minimal rank and triangular minimal rank differ, we might expect that the fractional minimal rank lies somewhere strictly in between.  For this reason, we looked at the fractional minimal rank of a partial matrix $\s A$ whose associated bipartite graph is a cycle on $2n$ vertices, which is the minimal instance of a graph that fails to be bipartite chordal.

\subsection{Open Problems}

Here are some conjectures that remain unproven and questions that remain unanswered, which we find interesting.

\begin{enumerate}

\item Regarding theorem (\ref{tiangen}), we believe that the result should hold over any field $\F$, whether or not it is algebraically closed.  In \cite{Tian}, we see that the result holds over any field when $n = 3$.

\item Our motivating conjecture from \cite{CJRW89}, which states that if $\s A$ is a graph whose associated graph is bipartite chordal, then $\tmr(\s A) = \mr(\s A)$, remains undecided.  Of course, if $\tmr(\s A) = \mr(\s A)$, then $\tmr(\s A) = \fmr(\s A) = \mr(\s A)$.

\item It is notable that in our investigation of this particular patterns of partial matrices, we found that $\tmr(\s A) \neq \mr(\s A)$ occurs if and only if $\tmr (\s A) < \fmr(\s A) < \mr(\s A)$. We conjecture that this should be true in general.  That is, for all partial matrices $\s A$, we have either $\tmr(\s A) < \fmr(\s A) < \mr(\s A)$ or $\tmr(\s A) = \fmr(\s A) = \mr(\s A)$.

\item Another question to be answered: in equation (\ref{fmrinf}), we characterized the $\fmr$ as an infimum over $b \in \N$.  Is this infimum necessarily attained?
For the s $\fmr(\s A)$ necessarily rational?

\item What does the set of minimizers look like?  In particular, the set of matrices $A \in \s A \otimes I_b$ satisfying $\rk(A) = \mr_b(\s A)$ form an affine algebraic set.  What are the properties of this set?  Perhaps something can be said about its dimension.

%Examples with $\fmr \in \N$ and $\tmr < \fmr < \mr$?

% ANOTHER QUESTION:  What are the multiplicity of these minimizers?  What do these solution sets look like as varieties?

\end{enumerate}

%\bibliographystyle{plain}
%\bibliography{master}

\end{document}